\documentclass[12pt,a4paper]{amsart}
\usepackage{a4wide}
\usepackage{amsmath, amssymb, amsfonts,enumerate}
\usepackage[all]{xy}
\usepackage{amscd}
\usepackage{comment}
\usepackage{mathtools}
\usepackage{tikz-cd} 
\usepackage{url}

\newtheorem{theorem}{Theorem}[section]
\newtheorem{lemma}[theorem]{Lemma}

 \theoremstyle{definition}
 
 \newtheorem{remark}[theorem]{Remark}

\numberwithin{equation}{section}
\newcommand {\Z}{\mathbb{Z}} 
\newcommand {\Q}{\mathbb{Q}} 
\newcommand {\C}{\mathbb{C}} 

\DeclareMathOperator{\Mat}{Mat}

\begin{document}
\title[First-order model theory and Kaplansky's conjecture]{First-order model theory and Kaplansky's stable finiteness conjecture}
\author[T.Ceccherini-Silberstein]{Tullio Ceccherini-Silberstein}
\address{Universit\`a del Sannio, I-82100 Benevento, Italy}
\email{tullio.cs@sbai.uniroma1.it}
\author[M.Coornaert]{Michel Coornaert}
\address{Universit\'e de Strasbourg, CNRS, IRMA UMR 7501, F-67000 Strasbourg, France}
\email{michel.coornaert@math.unistra.fr}
\author[X.K.Phung]{Xuan Kien Phung}
\address{D\'epartement d'Informatique et de Recherche Op\'erationnelle, Universit\'e de Montr\'eal, Montr\'eal, Qu\'ebec, H3T 1J4, Canada.}
\address{D\'epartement de Math\'ematiques et de Statistique, Universit\'e de Montr\'eal, Montr\'eal, Qu\'ebec, H3T 1J4, Canada.} 
\email{phungxuankien1@gmail.com}   
\begin{abstract}
Using algebraic geometry methods, the third author proved that the group ring of a surjunctive group with coefficients in a field is always stably finite.
In other words, every group satisfying Gottschalk's conjecture also satisfies Kaplansky's stable finiteness conjecture.
Here we present a  proof  of this result based on  first-order model theory.
\end{abstract}
\date{\today}
\maketitle

\setcounter{tocdepth}{1}
\tableofcontents

\section{Introduction}
A group $G$ is called \emph{surjunctive} if, for any finite set $A$, every injective  $G$-equivariant map $\tau \colon A^G \to A^G$ which is  continuous with respect to the prodiscrete topology on $A^G$ is surjective
(see Section~\ref{background} for more precise definitions). 
Since every injective self-mapping  of a finite set is  surjective, it is clear that all finite groups are surjunctive. 
More generally, Gromov~\cite{gromov-esav} and Weiss~\cite{weiss-sgds} proved that every sofic group is surjunctive. 
A conjecture going back to Gottschalk~\cite{gottschalk} is that every group is surjunctive.
Although Gottschalk's conjecture is believed to be false by several experts,
no example of a non-surjunctive group, not even of a non-sofic group,  has been found up to now.
\par
A ring $R$ is said to be \emph{stably finite} if, for any integer $d \geq 1$, every one-sided invertible square matrix of order $d$ with entries in $R$ is two-sided invertible.
Kaplansky~\cite[p.~122]{Kap2}, \cite[Problem~23]{Kap1} proved, using techniques from the theory of operator algebras,  that  the group ring
$K[G]$ is stably finite for any group $G$ and any field $K$ of characteristic $0$. Kaplansky also  asked whether this property remains true for fields of positive characteristic. 
The claim that $K[G]$ is stably finite for any group $G$ and any field $K$
is  known as ``Kaplansky's stable finiteness conjecture".
Elek and Szab{\'o}~\cite[Corollary~4.7]{es-direct} proved that every sofic group satisfies Kaplansky's stable finiteness conjecture
(see also~\cite[Corollary~1.4]{cc-israel},  \cite[Corollary~7.7]{g-p-linear-sofic}, and \cite[Corollary~1.9]{bradford2022hopfian}).   
\par
In~\cite[Theorem~B]{phung-pams-2023}, the third author of the present paper obtained the following result as a consequence of a general direct finiteness property of algebraic cellular automata.

\begin{theorem}
\label{t:main}
Let $G$ be a surjunctive group and let $K$ be a field. Then the group ring $K[G]$ is stably finite.
\end{theorem}

Observe that Theorem~\ref{t:main}, combined with the theorem of Gromov and Weiss on the surjunctivity of sofic groups, yields the theorem of Elek and Szab{\'o on the stable finiteness of group rings of sofic groups with coefficients in arbitrary fields.
\par
Several alternative proofs of Theorem~\ref{t:main} are known (see Remark~\ref{r:other-proofs} below).
The goal of the present  note   is to provide a   proof 
relying   on  model theory. Our proof does not make use of Kaplansky's result on the stable finiteness of group rings in  characteristic $0$.

\par

\noindent
{\bf Acknowledgements.} 
We are very grateful to Goulnara Arzhantseva, Henry Bradford, Vladimir Pestov, Benjamin Steinberg,  and Andreas Thom for
their interest in our work and comments. We would also like to thank  
Anand Pillay for helpful discussions concerning model theory.

\section{Background material}
\label{background}

\subsection{Gottschalk's conjecture and cellular automata} (See~\cite[Chapter~1]{ca-and-groups-springer})
Let $G$ be a group and let $A$ be a set.
Consider the set $A^G$ consisting of all maps $x \colon G \to A$.
A \emph{cellular automaton} over the group $G$ and the \emph{alphabet} $A$
is a map $\tau \colon A^G \to A^G$ satisfying the following property:
there exist a finite subset $S \subset G$
and a map $\mu \colon A^S \to A$ such that
\begin{equation}
\label{def:automate}
 (\tau(x))(g) = \mu((x \circ L_g)\vert_S)
\end{equation}
for all $x \in A^G$ and $g \in G$, where $L_g \colon G \to G$ denotes the left multiplication by $g$ and $(x \circ L_g)\vert_S$ denotes the restriction of  $x \circ L_g$ to $S$.
\par
The \emph{prodiscrete uniform structure} on $A^G$ is the product uniform structure on $A^G$ obtained by taking the uniform discrete structure on every factor $A$ of $A^G = \prod_{g \in G} A$.
The \emph{prodiscrete topology} on $A^G$ is the topology associated with the prodiscrete uniform structure on $A^G$.
The prodiscrete topology is also the product topology obtained by taking the discrete topology on every factor of $A^G$.
Equip $A^G$ with the left action of $G$ defined by $(g,x) \mapsto x \circ L_g$ for all $g \in G$ and $x \in A^G$.
A map $\tau \colon A^G \to A^G$ is a cellular automaton if and only if it is $G$-equivariant and uniformly continuous with respect to the prodiscrete uniform structure on $A^G$. 
In the case when $A$ is a finite set, a map $\tau \colon A^G \to A^G$ is a cellular automaton if and only if it is $G$-equivariant and continuous
with respect to the prodiscrete topology on $A^G$.
Thus, Gottschalk's conjecture amounts to saying that, for any group $G$ and any finite set $A$, every  injective cellular automaton $\tau \colon A^G \to A^G$ is surjective. 

\subsection{Stably finite rings} 
(See~\cite[Section~1.B]{lam-modules-rings})
All rings are assumed to be associative and unital.
The zero element of a ring is denoted by $0$ and its unital element  is denoted by $1$.
\par
A ring $R$ is called \emph{directly finite} if $a b = 1$ implies $b a = 1$ for all $a,b\in R$.
If $K$ is  a field and $V$ is a  vector space over $K$, 
then the endomorphism ring of $V$ is  directly finite if and only if $V$ is finite-dimensional.
A ring $R$ is called \emph{stably finite} if the ring $\Mat_d(R)$ of $d \times d$ matrices with entries in $R$ is stably finite for every integer $d \geq 1$.
Every stably finite ring $R$ is directly finite since the ring $\Mat_1(R)$ is isomorphic to  $R$.
All finite rings, all commutative rings, all fields, all division rings,
all one-sided Noetherian rings,  and all unit-regular rings  
are stably finite (and therefore directly finite).
There exist directly finite rings that are not stably finite (see for instance~\cite[Exercise~1.18]{lam-emr}).

\subsection{Group rings} (See~\cite{passman-group-rings})
Let $G$ be a group and let $K$ be a field.
The set $K^G$, which consists of all maps $\alpha \colon G \to K$, has a natural structure of a vector space over $K$.
The \emph{support} of $\alpha \in K^G$ is the subset of $G$ consisting of all $g \in G$ such that $\alpha(g) \not= 0$. 
Let  $K[G]$ denote the vector subspace of $K^G$ consisting of all   $\alpha \in K[G]$ having finite support.
The \emph{convolution product} $\alpha \beta$  of two elements $\alpha, \beta \in K[G]$ is defined by
\begin{equation*}\label{e;convolutional-prod}
(\alpha\beta)(g) \coloneqq  \sum_{\substack{h_1, h_2 \in G\\ h_1 h_2 = g}}\alpha(h_1)\beta(h_2)
\end{equation*}
for all $g \in G$.
Equipped with the convolution product, the vector space $K[G]$ is a $K$-algebra.
For $g \in G$, consider the element $\delta_g \in K[G]$ defined by $\delta_g(g) = 1$ and $\delta_g(h) = 0$ for all $h \in G \setminus \{g\}$.
Then $\delta_{1_G} = 1 \in K[G]$.
As $\delta_{g g'} = \delta_g \delta_{g'}$ for all $g,g' \in G$, we deduce that the map $g \mapsto \delta_g$ defines a group embedding of $G$ into the group of units of $K[G]$.
We have $\alpha = \sum_{g \in G} \alpha(g) \delta_g$ for all $\alpha \in K[G]$, so that the family $(\delta_g)_{g \in G}$ is a vector basis for $K[G]$.
\par
The \emph{group ring} of $G$ with coefficients in $K$ is the ring underlying the $K$-algebra $K[G]$.

 \subsection{Linear cellular automata and stable finiteness of group rings} (See~\cite{cc-israel}, \cite[Chapter~8]{ca-and-groups-springer}, )
Let $K$ be a field and let $V$ be a vector space over $K$.
The  set $V^G$ has a natural product structure of a vector space over $K$.
A cellular automaton $\tau \colon V^G \to V^G$ is called a \emph{linear cellular automaton} if $\tau$ is a $K$-linear map.
The stable finiteness of group rings admits the following interpretation in terms of linear cellular automata.

\begin{theorem}
\label{t:st-lca}
Let $G$ be a group and let $K$ be a field.
Then the ring $K[G]$ is stably finite if and only if, for any finite-dimensional vector space $V$ over $K$,
every injective linear cellular automaton $\tau \colon V^G \to V^G$ is surjective.
\end{theorem}

\begin{proof}
See \cite[Corollary~8.15.6]{ca-and-groups-springer}.
\end{proof}

\subsection{Model theory of algebraically closed fields} (See~\cite{marker-model-theory-gtm},  \cite[Section~5]{gromov-esav}) 
Two fields are called \emph{elementary equivalent} if they satisfy the same first-order sentences (i.e., first-order formulae without free variables in the language of rings).
Isomorphic fields are always elementary equivalent but the converse does not hold in general.
For example, the algebraic closure $\overline{\Q}$ of $\Q$ and the field $\C$ of complex numbers are not isomorphic since $\overline{\Q}$ is countable while $\C$ is uncountable.
However, the fields $\overline{\Q}$ and $\C$ are elementary equivalent
by the \emph{Lefschetz principle} whose general formulation is as follows:

\begin{theorem}[Lefschetz principle]
\label{t:el-equiv-fields}
Any two algebraically closed fields of the same characteristic are elementary equivalent.
\end{theorem}

\begin{proof}
This is a classical result in model theory which can be rephrased by saying that the theory of algebraically closed fields of a fixed characteristic is complete 
(see  Proposition~2.2.5 and Theorem~2.2.6 in~\cite{marker-model-theory-gtm}). 
\end{proof}

The Lefschetz principle is one of the key ingredients in our proof.
We shall also make use of the following:

\begin{theorem}
\label{t:char-p-char-0}
Let $\psi$ be a first-order sentence in the language of rings which is satisfied by some (and therefore any) algebraically closed field of characteristic $0$.
Then there exists an integer $N$ such that $\psi$ is satisfied by any 
algebraically closed field of  characteristic $p \geq N$.
\end{theorem}

\begin{proof}
This is (iii) $\implies$ (v) in \cite[Corollary~2.2.10]{marker-model-theory-gtm}.
\end{proof}

\section{Proof of Theorem~\ref{t:main}} 

Let us first establish some auxilliary results.

\begin{lemma}
\label{l:non-sf-first-order}
Let $G$ be a group, $d \geq 1$ an integer, and $S \subset G$ a finite subset.
Then there exists a first-order sentence $\psi_{d,S}$ in the language of rings such that a field $K$ satisfies $\psi_{d,S}$ if and only if 
there exist two matrices $A,B \in \Mat_d(K[G])$ such that
\begin{enumerate}[\rm (1)]
\item
the support of each entry of $A$ and $B$ is contained in $S$;
\item
$A B = 1$ and $B A \not= 1$.
\end{enumerate} 
\end{lemma}

\begin{proof}
Since $d$ and $S$ are fixed, we can quantify over $d\times d$ matrices in $\Mat_d(k[G])$ whose support of each entry is contained in $S$ by quantifying over the coefficients of every entry of the matrix. Consequently, the existence of two matrices $A,B \in \Mat_d(K[G])$ satisfying (1) and  (2) can be expressed by a $2d^2|S|$-variables first-order sentence $\psi_{d,S}$ in the language of rings, depending
only on the group multiplication table of the elements in $S$. 
\par 
For the sake of completeness, we give below an explicit  formula for $\psi_{d,S}$. 
We represent the entries at the position $(i,j)$ of the matrices $A$ and $B$  by $\displaystyle \sum_{s \in S} x_{i, j,s} s$ and $\displaystyle \sum_{s \in S} 
y_{i, j,s} s$ respectively. For $ 1 \leq i,j \leq d$ and $g \in S^2 = \{st: s, t \in S\} \subset G$, let 
\[
P(i,j,g)\coloneqq \sum_{k = 1}^d \sum_{\substack{s,t \,\in S\\ st= g}} x_{i,k,s} y_{k,j,t} 
\quad \text{and} \quad  
Q(i,j,g) \coloneqq \sum_{k = 1}^d \sum_{\substack{s,t \,\in S\\ st = g}}y_{i,k,s} x_{k,j,t}. 
\] 
Let $D \coloneqq \{(i,i,1_G) : 1 \leq i \leq d\}$. 
Then the properties $AB=1$ and $BA=1$ can be respectively expressed by the  first-order formulae 
$P$ and $Q$, where 
\[
P= \left(\bigwedge\limits_{\substack{ (i,i,1_G) \in D}} P(i,i,1_G)=1 \right)
\land   \left(\bigwedge\limits_{\substack{g \in S^2, (i,j,g) \notin D}} P(i,j,g)=0\right) 
\]
\[
Q= \left(\bigwedge\limits_{\substack{ (i,i,1_G) \in D}} Q(i,i,1_G)=1 \right)
\land   \left(\bigwedge\limits_{\substack{g \in S^2, (i,j,g) \notin D}} Q(i,j,g)=0\right). 
\] 
Hence, we can take $ \psi_{d,S} \coloneqq  \exists\, x_{i,j,s}, y_{i,j,s} \,(1 \leq i,j \leq d, \, s \in S), P \land \neg Q$. 
\end{proof}

 \begin{lemma}
\label{l:sf-KG-elem-equiv}
Let $G$ be a group and   suppose that $K$ and $L$ are elementary equivalent fields.
Then $K[G]$ is stably finite if and only if $L[G]$ is stably finite.
\end{lemma}

\begin{proof}
As $K$ and $L$ play symmetric roles, it suffices to show that if $K[G]$ is not stably finite then $L[G]$ is not stably finite.
So, let us assume that $K[G]$ is not stably finite.
This means that there exist an integer $d \geq 1$ and two square matrices $A$ and $B$ of order $d$ with entries in $K[G]$ such that
$A B = 1$ and $B A \not= 1$.
If  $S \subset G$ is a finite subset containing the support of each entry of $A$ and $B$,
the field $K$ satisfies the sentence $\psi_{d,S}$ given by Lemma~\ref{l:non-sf-first-order}.
Since $K$ and $L$ are elementary equivalent by our hypothesis,
The  sentence $\psi_{d,S}$ is also satisfied by the field $L$.
Consequently, the group ring $L[G]$ is not stably finite either.
\end{proof}

The following result is Theorem~B in~\cite{phung-weakly-surjunctive}. 

\begin{lemma}
\label{l:k-finite-is enough}
Let $G$ be a group and suppose that the group ring $K[G]$ is stably finite for every finite field $K$.
Then the group ring $K[G]$ is stably finite for any field $K$.
\end{lemma}

\begin{proof}
We divide the proof into four cases.
\noindent
Case 1: $K$ is the algebraic closure of the field $F_p \coloneqq \Z/p\Z$ for some prime $p$.
For every integer $n \geq 1$, let $K_n$ denote the subfield of $K$ consisting of all roots of the polynomial $X^{p^{n!}} - X$.
In other words, denoting by $\phi \colon K \to K$ the Frobenius automorphism, $K_n$ is the subfield of $K$ consisting of all fixed points of $\phi^{n!}$. 
We have $K_n \subset K_{n + 1}$ for all $n \geq 1$ and $K = \bigcup_{n \geq 1}  K_n$.  
Moreover,  $K_n$ is a finite field (of cardinality $p^{n!}$) for every $n \geq 1$.
Let $A$ and $B$ be square matrices of order $d$ with entries in $K[G]$ such that $A B = 1$.
Then there exists $n_0 \geq 1$ such that all entries of $A$ and $B$ are in $K_{n_0}[G]$.
Since $K_{n_0}[G]$ is stably finite by our hypothesis, we deduce that  $B A = 1$.
This shows that $K[G]$ is stably finite.
\par
\noindent
Case 2: $K$ is an  algebraically closed field of  characteristic $p > 0$.
This follows from Case 1, Lemma~\ref{l:sf-KG-elem-equiv}, and Theorem~\ref{t:el-equiv-fields}.
\par
\noindent   
Case 3: $K$ is an  algebraically closed field of charactristic $0$.
Suppose by contradiction that $K[G]$ is not stably finite.
This means that $K$ satisfies the sentence $\psi_{d,S}$ given by Lemma~\ref{l:non-sf-first-order} for some integer $d \geq 1$ and some finite subset $S \subset G$.
By applying Theorem~\ref{t:char-p-char-0}, we deduce that there exists an integer $N \geq 1$ such that
$\psi_{d,S}$ is satisfied by any algebraically closed field of characteristic $p \geq N$.
This implies that $L[G]$ is not stably finite whenever $L$ is an  algebraically closed field of characteristic $p \geq N$, in contradiction with Case 1.
\par  
\noindent
Case 4: $K$ is an arbitrary field.
Let $\overline{K}$ denote the algebraic closure of $K$.
Then $\overline{K}[G]$ is stably finite by   Case 2 and Case 3.
As $K[G]$ is a subring of $\overline{K}[G]$, we deduce that $K[G]$ is itself stably finite.
\end{proof}

\begin{remark}
We could also deduce Case 3 from the theorem of Kaplansky mentioned above asserting that $K[G]$ is stably finite  for any group $G$ and any field $K$ of characteristic $0$.
\end{remark}

\begin{proof}[Proof of Theorem~\ref{t:main}]
Suppose first that the field $K$ is finite.
Let $V$ be a finite-dimensional vector space over $K$.
Then $V$ is finite (of cardinality $|V| = |K|^{\dim(V)}$).
Since $G$ is surjunctive, every injective cellular automaton $\tau \colon V^G \to V^G$ is surjective.
In particular, every injective linear cellular automaton $\tau \colon V^G \to V^G$ is surjective.
Therefore $K[G]$ is stably finite by Theorem~\ref{t:st-lca}.
\par
By applying Lemma~\ref{l:k-finite-is enough}, we conclude that $K[G]$ is stably finite for any field $K$.
\end{proof}

\section{Final remarks}

\begin{remark}
The following definition was introduced by the third author in~\cite{phung-weakly-surjunctive}.
A group $G$ is called \emph{linearly surjunctive} if for every finite dimensional vector space $A$ over a finite field $k$, all injective linear cellular automata $\tau \colon A^G \to A^G$ are  surjective.
Every surjunctive group is clearly linearly surjunctive
but there might exist linearly surjunctive groups that are not surjunctive.
Observe that the hypothesis that $G$ is linearly surjunctive is sufficient in the first part of the proof of Theorem~\ref{t:main}.   
Thus, the conclusion of Theorem~\ref{t:main} remains valid for all linearly surjunctive groups.
\end{remark}

\begin{remark}
\label{r:dykema-juschenko}
Kaplansky's \emph{direct finiteness conjecture} asserts that the ring $K[G]$ is directly finite for any group $G$ and any field $K$.
Since stable finiteness implies direct finiteness, Kaplansky's direct finiteness conjecture  is a weakening of Kaplansky's stable finiteness conjecture.
In~\cite[Theorem~2.2]{dykema-juschenko}, Dykema and Juschenko have shown that, for any field $K$ and any group $G$,
the ring $K[G]$ is stably finite if and only if 
the ring $K[G \times H]$ is directly finite for every finite group $H$.
It follows that Kaplansky's direct finiteness conjecture is in fact equivalent to
Kaplansky's stable finiteness conjecture.
Note that, as pointed out in \cite[page~12]{bradford2022hopfian}, the above equivalence was already
known to Passman (see \cite{passman-review}).
\end{remark}

\begin{remark}
\label{r:other-proofs}
In \cite[p.~10]{capraro-lupini},
using the fact that every field embeds in an ultraproduct of finite fields (an observation credited to Pestov),
the authors prove that, given a group $G$,
the ring $K[G]$ is directly finite for any field $K$ as soon as $K[G]$ is directly finite for any finite field $K$.
They then deduce that $K[G]$ is directly finite for any surjunctive group $G$ and any field $K$.
As every virtually surjunctive group is surjunctive (see \cite[Lemma 6]{arzhanseva-gal} and \cite[Exercise~3.26]{csc-exos}),
this last  result, combined with the result of Dykema and Juschenko mentioned in Remark~\ref{r:dykema-juschenko},   implies  Theorem~\ref{t:main}.
Recently, Bradford and Fournier-Facio \cite[Corollary~3.25]{bradford2022hopfian} gave another proof of
Theorem~\ref{t:main}.
Note that their proof makes use of Kaplansky's result on the stable finiteness of group rings in characteristic $0$.
Other alternative proofs for Theorem~\ref{t:main} have been  privately communicated to us by Benjamin Steinberg and Andreas Thom.  
\end{remark}

\begin{remark}
There are also three famous conjectures attributed to Kaplansky about the structure of group rings of torsion-free groups:
the unit conjecture, the zero-divisor conjecture, and the idempotents conjecture.
Kaplansky's unit (resp.~zero-divisor, resp.~idempotent) conjecture
asserts that, for every torsion-free group $G$ and any field $K$, the ring $K[G]$ has no non-trivial units (resp.~no zero-divisors, resp.~no non-trivial idempotents).
Note that if $K[G]$ has no non-trivial units then it has no zero-divisors, and that if $K[G]$ has no zero-divisors then it has no non-trivial idempotents.
Thus, 
any torsion-free group satisfying Kaplansky's unit conjecture also satisfies Kaplansky's zero-divisor conjecture
and any torsion-free group satisfying Kaplansky's zero-divisor conjecture also satisfies Kaplansky's idempotent conjecture.
On the other hand, if a ring $R$ has no non-trivial idempotents then it is directly finite (observe that if $a,b \in R$ satisfy $a b = 1$, then $b a$ is an idempotent).
Thus, any torsion-free group satisfying Kaplansky's idempotent conjecture also satisfies Kaplansky's direct finiteness conjecture.
Recently, Gardam~\cite{gardam-unit-conjecture} disproved  Kaplansky's unit conjecture by exhibiting a non-trivial unit in the group ring of the Promislow group with coefficients in $F_2$ 
(the Promislow group is the fundamental group of the unique flat $3$-dimensional closed manifold which is a real homology sphere).
  By replacing $\psi_{d,S}$ by a suitably chosen first-order sentence $\psi_S$ in the language of rings
  and adapting the proofs of the three lemmas in the present section,
  one deduces that, given a torsion-free group $G$,
  the ring $K[G]$ has no non-trivial units (resp.~no zero-divisors, resp.~no non-trivial idempotents) for any field $K$ as soon as this is true for any finite field
(see also \cite[Remark~3.14]{bradford2022hopfian}).  
     \end{remark}

\bibliographystyle{siam}
\bibliography{model}

\end{document}